\documentclass{article}

\usepackage{amsmath,amsthm,amssymb}

\def\N{\mathbb N}

\def\A{\mathcal A}
\def\P{\mathcal P}
\def\S{\mathcal S}
\def\M{\mathcal M}
\def\u{\mathbf u}
\def\t{\mathbf t}
\def\Prel{\mathcal{P}^{\mathrm{rel}}}
\def\Psirel{{\Psi}^\mathrm{rel}}
\def\AC{\rho^{\mathrm{ab}}}

\DeclareMathOperator{\sgn}{sgn}

\newtheorem{lemma}{Lemma}[section]
\newtheorem{proposition}[lemma]{Proposition}
\newtheorem{theorem}[lemma]{Theorem}
\newtheorem{corollary}[lemma]{Corollary}
\newtheorem{observation}[lemma]{Observation}

\theoremstyle{definition}
\newtheorem{definition}[lemma]{Definition}

\theoremstyle{remark}
\newtheorem{remark}[lemma]{Remark}

\begin{document}

\title{Abelian properties of Parry words}

\author{Ond\v{r}ej Turek
\bigskip \\
\normalsize{Nuclear Physics Institute} \\
\normalsize{Academy of Sciences of the Czech Republic} \\
\normalsize{250 68 \v Re\v z, Czech Republic} \\[3pt]
\small{and} \\[3pt]
\normalsize{Bogolyubov Laboratory of Theoretical Physics} \\
\normalsize{Joint Institute for Nuclear Research} \\
\normalsize{141980 Dubna, Russia} \\[4pt]
\normalsize{email: \texttt{o.turek@ujf.cas.cz}}
}

\maketitle
\begin{abstract}
Abelian complexity of a word $\mathbf{u}$ is a function that counts the number of pairwise non-abelian-equivalent factors of $\mathbf{u}$ of length $n$. We prove that for any $c$-balanced Parry word $\mathbf{u}$, the values of the abelian complexity function can be computed by a finite-state automaton. The proof is based on the notion of relative Parikh vectors. The approach works for any function $F(n)$ that can be expressed in terms of the set of relative Parikh vectors corresponding to the length $n$. For example, we show that the balance function of a $c$-balanced Parry word is computable by a finite-state automaton as well.
\end{abstract}

\section{Introduction}

Abelian complexity of a word $\u$ is a function $\AC_\u:\N\to\N$ that counts the number of pairwise non-abelian-equivalent factors of $\u$ of length $n$ \cite{RSZ}.
Although the notion is simple, explicit evaluation of $\AC_\u(n)$ for a given infinite word $\u$ is a complicated task. Let us recall two main approaches to the problem.

The first approach consists in deriving an explicit formula for the abelian complexity function. This is usually extremely difficult, therefore, nontrivial infinite words with a known expression for $\AC_\u(n)$ are very rare. There exist only a few such examples to date:
\begin{itemize}
\item Sturmian words: $\AC_\u(n)=2$ for all $n\in\N$, cf.~\cite{CoHe};
\item Thue--Morse word~\cite{RSZ};
\item a special ternary word constructed so that its abelian complexity satisfies $\AC_\u(n)=3$ for all $n\in\N$, cf.~\cite{RSZ};
\item quadratic Parry words~\cite{BBT};
\item the Tribonacci word $\t$ (the fixed point of $0\mapsto01$, $1\mapsto02$, $2\mapsto0$): a simple criterion to decide whether $\AC_\t(n)=3$ is known~\cite{RSZ2}, and recently also a formula allowing to evaluate $\AC_\t(n)$ in $\mathcal{O}(\log n)$ steps has been obtained~\cite{TuX};
\item let us mention also the paperfolding word $\mathbf{f}$, for which a finite set of recurrent relations that determine the function $\AC_\mathbf{f}(n)$ has been found~\cite{MR}.
\end{itemize}
Note that these examples are related to words over binary and ternary alphabets. To the best of our knowledge, no results have been achieved for infinite words over alphabets consisting of more than three letters.

Another approach, the most natural one, consists in calculating values $\AC_\u(n)$ from the definition. 
That is, one slides a window of size $n$ on a sufficiently long prefix of $\u$ and counts the classes of abelian-equivalent factors. Nevertheless, this is a brute-force method that can be used in practice only for small values of $n$. The length of the prefix that must be sought through is typically much greater than $n$, thus the calculation for large $n$ becomes extremely slow, and even when a powerful computer is used, it sooner or later fails for memory reasons.

In this paper we deal with an approach that is, in a way, a combination of the previous two ones. We show that for any $c$-balanced Parry word $\u$, values $\AC_\u(n)$ can be calculated by a finite-state automaton with a normal $U$-representation of $n$ as its input. In other words, instead of sliding a window of size $n$ on a certain prefix of $\u$, which is inconvenient because the required prefix length grows to infinity as $n\to\infty$, one performs a walk on a transition diagram of a discrete finite-state automaton, which is a \emph{finite} graph, independent of $n$. The result can be iterpreted also in the way that there exist functions $\delta$ and $\tau$ allowing to evaluate $\AC_\u(n)$ in $\mathcal{O}(\log n)$ steps. Our proof is constructive; we explain how to derive the finite-state automaton in question for a given word $\u$, i.e., we will explain how to find the functions $\delta$ and $\tau$.

\section{Preliminaries}\label{Preliminaries}

Let us consider a set $\A=\{0,1,2,\ldots,m-1\}$ (\emph{alphabet}) consisting of $m$ symbols (\emph{letters}) $0,1,\ldots,m-1$.
Concatenations of letters from $\A$ are called \emph{words}. Let $\A^*$ denote the free monoid of all finite words over $\A$ including the empty word $\epsilon$. The \emph{length} of a $w=w_0w_1w_2\cdots w_{n-1}\in\A^*$ is the number of its letters, $|w|=n$; the length of the empty word is defined to be $0$. The symbol $|w|_\ell$ for $\ell\in\A$ and $w\in\A^*$ denotes the number of occurences of the letter $\ell$ in the word $w$.

Infinite sequences of letters are called \emph{infinite words}. %The set of all infinite words over $\A$ is denoted by $\A^\N$.
A finite word $w$ is a \emph{factor} of a (finite or infinite) word $\u$ if there exists a finite word $x$ and a (finite or infinite, respectively) word $y$ such that $\u=xwy$. The word $w$ is called a \emph{prefix} of $\u$ if $x=\epsilon$, and a \emph{suffix} of $\u$ if $y=\epsilon$.

For every $w\in\A^*$ and $k\in\N$, the concatenation of $k$ words $w$ is denoted by $w^k$. We set $w^0=\epsilon$. If a word $v$ has a prefix $w^k$, we use the symbol $w^{-k}v$ to denote the word satisfying $w^kw^{-k}v=v$; the symbol $vw^{-k}$ is defined analogously.

An infinite word $\u$ is called \emph{recurrent} if every factor of $\u$ occurs infinitely many times in $\u$.

An infinite word $\u$ is said to be \emph{$c$-balanced} if for every $\ell\in\A$ and for every pair of factors $v$, $w$ of $\u$ such that $|v|=|w|$, it holds
$\left||v|_\ell-|w|_\ell\right|\leq c$.

The \emph{Parikh vector} of a factor $w$ is the $m$-tuple $\Psi(w)=(|w|_0,|w|_1,\ldots,|w|_{m-1})$; note that $|w|_0+|w|_1+\cdots+|w|_{m-1}=|w|$.
For any given infinite word $\u$, let $\mathcal{P}_\u(n)$ denote the set of all Parikh vectors corresponding to factors of $\u$ having the length $n$, i.e.,
$$
\mathcal{P}_\u(n)=\left\{\Psi(w)\,;\,\text{$w$ is a factor of $\u$}, |w|=n\right\}.
$$
The \emph{abelian complexity} of a word $\u$ is the function $\AC_\u:\N\to\N$ counting the number of elements of sets $\mathcal{P}_\u(n)$, i.e.,
\begin{equation}\label{AC}
\AC_\u(n)=\#\P_\u(n)\,,
\end{equation}
where $\#$ denotes the cardinality.

The \emph{relative Parikh vector}~\cite{Tu13} is defined for any factor $w$ of $\u$ of length $n$ as
\begin{equation}\label{Psi rel}
{\Psi}_\u^\mathrm{rel}(w)=\Psi(w)-\Psi(u_0u_1\cdots u_{n-1})\,.
\end{equation}
The sum of components of ${\Psi}_\u^\mathrm{rel}(w)$ is always equal to $0$. If moreover $\u$ is a $c$-balanced word, then the components of ${\Psi}_\u^\mathrm{rel}(w)$ are bounded by $c$ for any factor $w$ of $\u$, cf.~\cite{Tu13}. Therefore, the set of all relative Parikh vectors $\left\{{\Psi}_\u^\mathrm{rel}(w)\,;\,\text{$w$ is a factor of $\u$}\right\}$ is finite for any $c$-balanced word $\u$, which is a particularly important fact.

Since the subtrahend $\Psi(u_0u_1\cdots u_{n-1})$ on the right-hand side of~\eqref{Psi rel} depends only on $n$ (and does not depend on $w$), the set of relative Parikh vectors corresponding to the length $n$,
$$
\Prel_\u(n):=\left\{\left.{\Psi}_\u^\mathrm{rel}(w)\;\right|\;\text{$w$ is a factor of $\u$}, |w|=n\right\},
$$
has the same cardinality as the set of Parikh vectors, $\P_\u(n)$. Hence we obtain, with regard to~\eqref{AC}, the formula
\begin{equation}\label{ACrel}
\AC_\u(n)=\#\Prel_\u(n)\,.
\end{equation}

\emph{Parry words} are infinite words associated with the set of $\beta$-integers for $\beta$ being a Parry number.
The famous Fibonacci word and the Tribonacci word are examples of Parry words.
Parry words are divided into two classes:
\begin{itemize}
\item A \emph{simple Parry word} over $\A=\{0,1,\ldots,m-1\}$ is a fixed point of a substitution
\begin{equation}\label{simpleParry}
\begin{array}{rccl}
\varphi:\qquad & 0 & \mapsto & 0^{\alpha_0}1 \\
& 1 & \mapsto & 0^{\alpha_1}2 \\
&   & \vdots & \\
& m-2 & \mapsto & 0^{\alpha_{m-2}}(m-1) \\
& m-1 & \mapsto & 0^{\alpha_{m-1}}
\end{array}
\end{equation}
\item A \emph{non-simple Parry word} over $\A=\{0,1,\ldots,m+p-1\}$ is a fixed point of
\begin{equation}\label{nonsimpleParry}
\begin{array}{rccl}
\varphi:\qquad & 0 & \mapsto & 0^{\alpha_0}1 \\
& 1 & \mapsto & 0^{\alpha_1}2 \\
&   & \vdots & \\
& m & \mapsto & 0^{\alpha_{m}}(m+1) \\
&   & \vdots & \\
& m+p-2 & \mapsto & 0^{\alpha_{m+p-2}}(m+p-1) \\
& m+p-1 & \mapsto & 0^{\alpha_{m+p-1}}m
\end{array}
\end{equation}
\end{itemize}
The exponents $\alpha_j$ occurring in \eqref{simpleParry} and \eqref{nonsimpleParry} are non-negative integers obeying certain restrictions~\cite{Pa,Fa}. Both substitutions must satisfy $\alpha_0\geq1$ and $\alpha_\ell\leq \alpha_0$ for all $\ell\in\A$. In addition, substitution~\eqref{simpleParry} requires $\alpha_{m-1}\geq1$, whereas substitution \eqref{nonsimpleParry} requires $\alpha_\ell\geq1$ for a certain $\ell\in\{m,m+1,\ldots,m+p-1\}$.

For a given substitution~\eqref{simpleParry} or~\eqref{nonsimpleParry}, let us set $U_j=|\varphi^j(0)|$ for every $j\in\N_0$. 
Any $n\in\N$ can be represented as a sum
\begin{equation}\label{F-rep sum}
n=\sum_{j=0}^{k} d_j U_j
\end{equation}
with integer coefficients $d_j$.
If coefficients $d_j$ are obtained by the greedy algorithm, the sequence $d_kd_{k-1}\cdots d_1d_0$ is called \emph{normal $U$-representation} of $n$ \cite{Lo} and denoted
\begin{equation}\label{F-rep}
\langle n \rangle_U=d_kd_{k-1}\cdots d_1d_0\,.
\end{equation}
The greedy algorithm implies that the coefficients in~\eqref{F-rep} satisfy $d_j\in\{0,1,\ldots,\alpha_0\}$ for all $j=0,1,\ldots,k$.
If $\u$ is the fixed point of $\varphi$, the normal $U$-representation allows to express a prefix of $\u$ of given length $n$ \cite{Dum,Fa}. Namely, the prefix of $\u$ of length $n$ represented by $\langle n \rangle_U=d_kd_{k-1}\cdots d_1d_0$ takes the form
\begin{equation}\label{prefix u}
u_0u_1\cdots u_{n-1}=\left(\varphi^{k}(0)\right)^{d_{k}}\left(\varphi^{k-1}(0)\right)^{d_{k-1}}\cdots\left(\varphi(0)\right)^{d_1}0^{d_0}\,.
\end{equation}

The \emph{incidence matrix} $\M_\varphi$ of a substitution $\varphi$ on $\A=\{0,1,\ldots,m-1\}$ is defined by
$$
\M_\varphi=\begin{pmatrix}
|\varphi(0)|_0 & |\varphi(0)|_1 & \cdots & |\varphi(0)|_{m-1} \\
|\varphi(1)|_0 & |\varphi(1)|_1 & \cdots & |\varphi(1)|_{m-1} \\
\vdots & \vdots & & \vdots \\
|\varphi(m-1)|_0 & |\varphi(m-1)|_1 & \cdots & |\varphi(m-1)|_{m-1}
\end{pmatrix}\,.
$$
The notion of incidence matrix has several useful applications. It follows immediately from the definition of $\M_\varphi$ that for any $w\in\A^*$,
\begin{equation}\label{matice}
\Psi(\varphi(w))=\Psi(w)\M_\varphi\,.
\end{equation}
Furthermore, due to~\cite{adamczewski}, if all the eigenvalues of $\M_\varphi$ except the dominant one are of modulus less than one, then the fixed point of $\varphi$ is $c$-balanced for a certain $c$.

\medskip

A \emph{deterministic finite automaton with output} (DFAO) (cf.~\cite{AS}) is a $6$-tuple $(Q,\Sigma,\delta,q_0,\Delta,\tau)$, where $Q$ is a finite set of states, $\Sigma$ is the finite input alphabet, $\delta:Q\times\Sigma\to Q$ is the transition function, $q_0$ is the initial state, $\Delta$ is the output alphabet, and $\tau:Q\to\Delta$ is the output function. If we extend the domain of $\delta$ to $Q\times\Sigma^*$ by defining $\delta(q,\epsilon)=q$ for all $q\in Q$, and $\delta(q,xa)=\delta(\delta(q,x),a)$ for all $q\in Q$, $x\in\Sigma^*$ and $a\in\Sigma$, a DFAO defines a function $f:\Sigma^*\to\Delta$ given as
$$
f(w)=\tau(\delta(q_0,w)) \qquad \text{for $w\in\Sigma^*$}.
$$
A sequence $(a_n)_{n\in\N}$ with values in a finite alphabet $\Delta$ is called \emph{$U$-automatic} (cf.~\cite{Sh88}) if there exists a DFAO $(Q,\Sigma,\delta,q_0,\Delta,\tau)$ with $\Sigma=\{0,1,\ldots,\alpha_0\}$ such that
$$
a_n=\tau(\delta(q_0,\langle n\rangle_U)) \qquad \text{for all $n\in\N$}.
$$

\section{Abelian complexity of $c$-balanced Parry words}\label{Main}

Let $\u$ be a Parry word, i.e., the fixed point of a substitution \eqref{simpleParry} or \eqref{nonsimpleParry}.
\begin{observation}\label{unbalanced}
If $\u$ is not $c$-balanced for any $c>0$, then $\left(\AC_\u(n)\right)_{n=1}^\infty$ is not an automatic sequence.
\end{observation}
Observation~\ref{unbalanced} holds trivially, because any word that is not $c$-balanced has obviously unbounded abelian complexity function, which, consequently, cannot be evaluated by an automaton with a finite output alphabet.

In this section we prove the reverse implication, that is, if $\u$ is $c$-balanced for a certain $c>0$, then $\left(\AC_\u(n)\right)_{n=1}^\infty$ is a $U$-automatic sequence. We will present a constructive proof, in which we explicitly derive the finite automaton in question.

From now on until the end of the paper we assume that $\u$ is a fixed point of a substitution $\varphi$ of type~\eqref{simpleParry} or \eqref{nonsimpleParry}, which is moreover $c$-balanced for a certain $c>0$. For the sake of simplicity, we will drop the subscript $\varphi$ in the symbol $\M_\varphi$, as well as the subscript $\u$ in the symbols $\P_\u(n),\Prel_\u(n)$, $\Psi_\u^\mathrm{rel}(w)$ and $\AC_\u(n)$.

\subsection{The main idea}

We begin the exposition by sketching the key idea of our approach. Our strategy consists in the use of the assumptions for introducing certain finite sets $\S(n)$ for $n\in\N$ (their structure will be described below) with the following properties.
\begin{itemize}
\item[(P1)] For any $n\in\N$, the set of relative Parikh vectors $\Prel(n)$ can be constructed using the set $\S(n)$.
\item[(P2)] There exists a finite number of sets $\S_1,\S_2,\ldots,\S_M$ such that for any $n\in\N$, $\S(n)=\S_j$ for a certain $j\in\{1,2,\ldots,M\}$.
\item[(P3)] If the normal $U$-representation of a number $N\in\N$ satisfies $\langle N\rangle_U=\langle n\rangle_U d$ for certain $n\in\N$ and $d\in\{0,1,\ldots,\alpha_0\}$, then the set $\S(N)$ can be constructed from $\S(n)$.
\end{itemize}
Property (P2) combined with property (P1) guarantees the existence of finitely many sets of relative Parikh vectors, $\Prel_1,\ldots,\Prel_M$, such that $\S(n)=\S_j\Rightarrow\Prel(n)=\Prel_j$. At the same time, combining property (P2) with property (P3) allows us to define a function $\delta(j,d)$ such that $\bigl(\;\S(n)=\S_j\wedge\langle N\rangle_U=\langle n\rangle_U d\;\bigr)\Rightarrow\S(N)=\S_{\delta(j,d)}$.

Once the sets $\Prel_1,\ldots,\Prel_M$ are established, one can introduce a function $\tau:\{1,2,\ldots,M\}\to\N$ defined as $\tau(j)=\#\Prel_j$. Then the calculation of $\AC(n)$ for a given $n\in\N$ is carried out as follows. In the first step, the function $\delta$ is used to tranform $\langle n\rangle_U$ into the value $j$ such that $\S(n)=\S_j$. Note that $j$ can attain only values $1,\ldots,M$, thus a machine with \emph{finitely many states} is sufficient to perform the procedure. In the second step, the function $\tau$ is used to transform the value $j$ into the value $\AC(n)$. Note that it holds $\S(n)=\S_j \;\Rightarrow\; \Prel(n)=\P_j \;\Rightarrow\; \AC(n)=\tau(j)$, cf. equation~\eqref{ACrel},

For the sake of clarity, the section is divided into subsections according to the following outline. At first we define the sets $\S(n)$. Then we prove, step by step, that the sets $\S(n)$ have properties (P1), (P2), (P3). Finally, we summarize the facts and formulate the main result, i.e., we express $\AC(n)$ in terms of $\langle n\rangle_U$, and we state that the sequence $\left(\AC(n)\right)_{n=1}^{\infty}$ is $U$-automatic.

\subsection{Definition of $\S(n)$}

Establishing the sets $\S(n)$ for $n\in\N$ is the initial step. However,
before we proceed to the definition of $\S(n)$, we need to introduce two auxiliary constants, which will be denoted by $H$ and $L$.
For any finite factor $w$ of $\u$, let $h_w$ be the sum of components of the vector $\Psirel(w)\M$. Since $\u$ is $c$-balanced by assumption, the set $\{\Psirel(w)\,;\,\text{$w$ is a factor of $\u$}\}$ is finite (see Sect.~\ref{Preliminaries}), hence the set $\{|h_w|\,;\,\text{$w$ is a factor of $\u$}\}$ is finite as well. Therefore, it has a maximum. We put $H$ to be any (fixed) number satisfying
\begin{equation}\label{H}
H\geq\max\{|h_w|\,;\,\text{$w$ is a factor of $\u$}\}\,.
\end{equation}
If $\varphi$ is a substitution \eqref{simpleParry} or \eqref{nonsimpleParry}, it holds $|\varphi(w)|=(\alpha_0+1)|w|_0+\sum_{\ell=1}^{\#\A-1}|\varphi(\ell)|\cdot|w|_\ell\geq\alpha_0|w|_0+|w|$ for any $w\in\A^*$. Since Parry words are recurrent, we have moreover $|w|\to\infty\Rightarrow|w|_0\to\infty$. To sum up, $|\varphi(w)|-|w|\geq\alpha_0|w|_0\to\infty$ as $|w|\to\infty$. Consequently, for any constant $C$ there exists a length $L$ such that the inequality $|\varphi(w)|-|w|\geq C$ holds true for all factors $w$ of $\u$ satisfying $|w|\geq L$. For technical reasons we consider the value $C=2\alpha_0+H$ and fix $L$ to be any number such that the implication
\begin{equation}\label{L}
|w|\geq L\quad\Rightarrow\quad|\varphi(w)|-|w|\geq2\alpha_0+H
\end{equation}
holds true for all factors $w$ of $\u$.
\begin{remark}\label{HL}
Fixing $H$ and $L$ does not require any detailed information on the structure of factors of $\u$. The knowledge of a balance bound $c$ and a certain prefix of $\u$ are sufficient. Indeed, one can set for example $H$ to be the maximal sum of components of vectors $(c_0,c_1,\ldots,c_{\#\A-1})\M$, where $c_j\in\{-c,-c+1,\ldots,c-1,c\}$ for all $j\in\{0,1,\ldots,\#\A-1\}$, and
$$
L=\min\left\{n\;;\;|u_0u_1\cdots u_{n-1}|_0\geq2+\frac{H}{\alpha_0}+c\right\}\,.
$$
(Since these formulas are rather illustrative and not essential for our further considerations, we omit the proof.)
\end{remark}

With the constant $L$ in hand, we can introduce sets $\S(n)$ for $n\in\N$.

\begin{definition}
Let $L$ be the number introduced by equation~\eqref{L}. For all $n\in\N$, we define the set
\begin{equation}\label{S(n)}
\S(n)=
\left\{
\left(
\Psirel(u_j u_{j+1}\cdots u_{j+n-1}),
u_j,
u_{j+n-L}\cdots u_{j+n+L}
\right)
\;;\; j\geq L
\right\}\,.
\end{equation}
\end{definition}

\begin{remark}\label{Rem. S(n)}
The set $\S(n)$ consists of triples $(\psi,a,b_{-L}\cdots b_0\cdots b_L)$, where
\begin{itemize}
\item $\psi=\Psi(w)$ is the Parikh vector of a certain factor $w$ of $\u$ of length $n$;
\item $a\in\A$ is the first letter of $w$;
\item $b_{-L}\cdots b_{L}$ is a factor of $\u$ of length $2L+1$; its middle letter $b_0$ coincides with the successor of the last letter of $w$ in $\u$.
\end{itemize}
\end{remark}
The relative positions of $w$, $a$ and $b_{-L}\cdots b_0\cdots b_L$ in $\u$ can be illustrated in the following way:
$$
\u=u_0\cdots u_{j-1}\underbrace{\overbrace{u_j}^{a} u_{j+1}\cdots \overbrace{u_{j+n-L}\cdots u_{j+n-1}}^{b_{-L}\cdots b_{-1}}}_{w}\overbrace{u_{j+n}}^{b_0}\overbrace{u_{j+n+1}\cdots u_{j+n+L}}^{b_1\cdots b_L}u_{j+n+L+1}\cdots\,.
$$
Since $j$ takes all values starting with $L$, the factor $u_j u_{j+1}\cdots u_{j+n-1}$ in equation~\eqref{S(n)} scans all factors of length $n$ occurring in $\u$ except for the prefix of $\u$ of length $L-1$.

\begin{remark}
It is possible to formulate the definition~\eqref{S(n)} using factors of the type $u_{j+n-L_1}\cdots u_{j+n+L_2}$ for two independent constants $L_1$ and $L_2$ instead of $u_{j+n-L}\cdots u_{j+n+L}$. Involving two constants allows to choose them such that $L_1+L_2<2L$, i.e., the factor $u_{j+n-L_1}\cdots u_{j+n+L_2}$ can be taken shorter than $u_{j+n-L}\cdots u_{j+n+L}$. This improvement leads to a more efficient calculation. However, we stick to using a single constant $L$ in order to simplify the exposition.
\end{remark}

\subsection{Property (P1)}

Let us show that the set or relative Parikh vectors $\Prel(n)$ can be trivially obtained from $\S(n)$.
\begin{observation}
For all $n\in\N$, it holds
\begin{equation}\label{Prel S}
\Prel(n)=\left\{\psi\;;\;(\psi,a,b_{-L}\cdots b_0\cdots b_L)\in\S(n)\right\}.
\end{equation}
\end{observation}

\begin{proof}
Due to the definition of $\S(n)$, the right-hand side of~\eqref{Prel S} satisfies
\begin{equation}\label{psi S}
\left\{\psi\;;\;(\psi,a,b_{-L}\cdots b_0\cdots b_L)\in\S(n)\right\}=\left\{\Psirel(u_j u_{j+1}\cdots u_{j+n-1}) \; ; \; j\geq L\right\}.
\end{equation}
By definition of $\Prel(n)$, we have
$$
\Prel(n)=\left\{\Psirel(u_j u_{j+1}\cdots u_{j+n-1}) \; ; \; j\in\N_0\right\}.
$$
However, since $\u$ is a Parry word and, therefore, a recurrent word, it also holds
$$
\Prel(n)=\left\{\Psirel(u_j u_{j+1}\cdots u_{j+n-1}) \; ; \; j\geq C'\right\}
$$
for any $C'\in\N_0$, in particular for $C'=L$. A comparison of this equation for $C'=L$ with equation~\eqref{psi S} gives equation~\eqref{Prel S}.
\end{proof}

\subsection{Property (P2)}\label{Sect. (P2)}

The proof of property (P2) is done in two easy steps. At first we take advantage of the $c$-balancedness of $\u$ in Proposition~\ref{U S(n) is finite} below.
\begin{proposition}\label{U S(n) is finite}
The union $\bigcup_{n=1}^\infty\S(n)$ is a finite set.
\end{proposition}

\begin{proof}
If $\u$ is defined over an alphabet $\A$, there obviously exist at most $(\#\A)^{2L+2}$ couples of the type $(u_j,u_{j+n-L}\cdots u_{j+n+L})$. Furthermore, since $\u$ is assumed to be $c$-balanced for a certain $c$, the entries of $\Psi(w)$ are bounded by $c$ for any factor $w$ of $\u$. Consequently, the number of all relative Parikh vectors of factors of $\u$ is bounded by $(2c+1)^{\#\A}$. To sum up, $\#\bigcup_{n=1}^\infty\S(n)\leq(2c+1)^{\#\A}\cdot (\#\A)^{2L+2}$.
\end{proof}

Property (P2) is a straightforward corollary of Proposition~\ref{U S(n) is finite}.
\begin{corollary}\label{Finite}
There exist sets $\S_1,\S_2,\ldots,\S_{M}$ such that
\begin{equation}\label{S(n)=S_j}
\left(\forall n\in\N\right)\,\left(\exists j\in\{1,2,\ldots,M\}\right)\,\left(\S(n)=\S_j\right).
\end{equation}
\end{corollary}

\begin{proof}
Each $\S(n)$ is a subset of $\bigcup_{n=1}^\infty\S(n)$. The union $\bigcup_{n=1}^\infty\S(n)$ is finite due to Proposition~\ref{U S(n) is finite}, thus there can exist only finitely many its subsets.
\end{proof}

\subsection{Property (P3)}\label{Sect. (P3)}

The proof of property (P3) begins with an auxiliary proposition, the aim of which is to show that $\S(n)$ can be found by exploring just a certain specified finite segment of $\u$.
\begin{proposition}\label{B(n)}
There exist constants $P,Q\in\N$ such that for any $n\in\N$, the set $\S(n)$ is given by
\begin{equation}\label{S(n) p q}
\S(n)=
\left\{
\left(
\Psirel(u_j u_{j+1}\cdots u_{j+n-1}),
u_j,
u_{j+n-L}\cdots u_{j+n+L}
\right)
\;;\; U_{k+P}\leq j<U_{k+Q}
\right\},
\end{equation}
where $k\in\N_0$ is chosen so that $n<U_{k+1}$.
\end{proposition}

\begin{proof}
First of all, let us make clear that for each $j\geq L$, the triple
$$
\left(\Psirel(u_j u_{j+1}\cdots u_{j+n-1}),u_j,u_{j+n-L}\cdots u_{j+n+L}\right)\in\S(n)
$$
can be unambiguously constructed from the factor $u_{j-L+1}\cdots u_{j+n+L}$. Indeed,
\begin{itemize}
\item $u_{j+n-L}\cdots u_{j+n+L}$ is a suffix of $u_{j-L+1}\cdots u_{j+n+L}$ of length $2L+1$,
\item $u_j$ is the $L$-th letter of $u_{j-L+1}\cdots u_{j+n+L}$,
\item $\Psirel(u_j u_{j+1}\cdots u_{j+n-1})$ is the relative Parikh vector of a factor that begins at the $L$-th letter of $u_{j-L+1}\cdots u_{j+n+L}$ and ends at the $(L+n-1)$-th letter of $u_{j-L+1}\cdots u_{j+n+L}$.
\end{itemize}
Note that the factor $u_{j-L+1}\cdots u_{j+n+L}$ has length $n+2L$. To sum up, any element of $\S(n)$ can be constructed from a certain factor of $\u$ of length $n+2L$. Let us find $P$ and $Q$ such that the set $\{u_{j-L+1}\cdots u_{j+n+L}\,;\,U_{k+P}\leq j<U_{k+Q}\}$ contains all factors of $\u$ of length $n+2L$. Obviously, once such constants are found, it follows that the set
$$
\{\left(\Psirel(u_j u_{j+1}\cdots u_{j+n-1}),u_j,u_{j+n-L}\cdots u_{j+n+L}\right)\;;\; U_{k+P}\leq j<U_{k+Q}\}
$$
is equal to the whole set $\S(n)$.

The derivation of $P$ and $Q$ will require a convenient estimate of $n+2L$. We have $n<U_{k+1}$ by assumption. Since $L$ is independent of $n$ and the values $U_i$ grow in general roughly exponentially with $i$, there exists a constant $r$ (independent of $n$ and $k$) with the property $n+2L<U_{k+1+r}$.

At the same time, according to~\cite[Prop.~4.8]{Tu13}, there exists a constant $R$ such that all factors of $\u$ of length $n'<U_{k'}$ can be found in the prefix of $\u$ of length $n'+U_{R+k'}$. Obviously, the constant $R$ can be assumed big enough so that all factors of $\u$ of length $n'$ can be found also in the prefix of $\u$ of length $U_{R+k'+1}$, i.e., in the word $\varphi^{R+k'+1}(0)$. The special choice $n'=n+2L$ and $k'=k+1+r$ then gives: All factors of $\u$ of length $n+2L$ are contained in $\varphi^{R+k+r+2}(0)$.
Now we are ready to establish $P$ and $Q$.

\begin{enumerate}
\item
Let $P$ be an integer satisfying $P\geq R+r+2$ and $U_P\geq L$. The condition $P\geq R+r+2$ is technical (needed for the corectness of the definition of $Q$ below), the condition $U_P\geq L$ ensures that the factor $u_{j+n-L}\cdots u_{j+n+L}$ is well defined for all $j\geq U_{k+P}$, because $L\leq U_P$ and $j\geq U_{k+P}$ trivially implies $j+n-L\geq0$.
\item
Having $P$ fixed, let us set $Q$ to be a number such that $\left(\varphi^{P-R-r-2}(0)\right)^{-1}\varphi^{Q-R-r-2}(0)$ contains a letter $0$.
\end{enumerate}
The choice of $Q$ ensures that the word
$$
\left(\varphi^{k+P}(0)\right)^{-1}\varphi^{k+Q}(0)=\varphi^{k+R+r+2}\left(\left(\varphi^{P-R-r-2}(0)\right)^{-1}\varphi^{Q-R-r-2}(0)\right)
$$
has the factor $\varphi^{R+k+r+2}(0)$.
Since we already know that $\varphi^{R+k+r+2}(0)$ contains all factors of $\u$ of length $n+2L$, we infer that $\left(\varphi^{k+P}(0)\right)^{-1}\varphi^{k+Q}(0)$ contains all factors of $\u$ of length $n+2L$ as well.
This fact together with the inclusion
\begin{multline*}
\{w\,;\,\text{$w$ is a factor of $\left(\varphi^{k+P}(0)\right)^{-1}\varphi^{k+Q}(0)$ of length $n+2L$}\} \\
=\{u_{j-L+1}\cdots u_{j+n+L}\,;\,U_{k+P}+L-1\leq j<U_{k+Q}-n-L\} \\
\subset\{u_{j-L+1}\cdots u_{j+n+L}\,;\,U_{k+P}\leq j<U_{k+Q}\}\,,
\end{multline*}
implies that $\{u_{j-L+1}\cdots u_{j+n+L}\,;\,U_{k+P}\leq j<U_{k+Q}\}$ indeed contains all factors of $\u$ of length $n+2L$, as we set to prove.
\end{proof}

In Proposition~\ref{inductive} below, we demonstrate that sets $\S(n)$ can be constructed in an inductive way, using the normal $U$-representation of $n$. This method is considerably more efficient for obtaining $\S(n)$ for large $n$ than using formula~\eqref{S(n) p q}. The result will be also essential for proving the $U$-automaticity of $\left(\AC(n)\right)_{n=1}^\infty$.

\begin{proposition}\label{inductive}
There exists an algorithm transforming the set $\S(n)$ into the set $\S(N)$
for any pair of integers $n,N\in\N$ such that $\langle N\rangle_U=\langle n\rangle_U d$ for a certain $d\in\{0,1,\ldots,\alpha_0\}$, i.e.,
\begin{equation}\label{Nnq}
\langle n\rangle_U=d_{k}\cdots d_0\,, \qquad
\langle N\rangle_U=d_k\cdots d_{0}d\,.
\end{equation}
\end{proposition}

\begin{proof}
Since $\langle N\rangle_U$ has $k+2$ digits, the greedy algorithm implies that $N<U_{k+2}$. Therefore, according to Proposition~\ref{B(n)}, the set $\S(N)$ is given as
$$
\S(N)=
\left\{
\left(
\Psirel(u_J u_{J+1}\cdots u_{J+N-1}),
u_J,
u_{J+N-L}\cdots u_{J+N+L}
\right)
\;;\; U_{k+1+P}\leq J<U_{k+1+Q}
\right\}\,.
$$
Let $\left(\Psirel(u_J\cdots u_{J+N-1}),u_J,u_{J+N-L}\cdots u_{J+N+L}\right)\in\S(N)$.
In the proof we will find an element $\left(\psi,a,b_{-L}\cdots b_L\right)\in\S(n)$ and a way how to express
$$
\left(\Psirel(u_J\cdots u_{J+N-1}),u_J,u_{J+N-L}\cdots u_{J+N+L}\right)
$$
in terms of $\left(\psi,a,b_{-L}\cdots b_L\right)$. On the other hand, it will be obvious that for a given $\left(\psi,a,b_{-L}\cdots b_L\right)\in\S(n)$, the algorithm gives an element of $\S(N)$ (more precisely speaking, one element of $\S(n)$ leads generally to several elements of $\S(N)$; details will be explained later). To sum up, the method we are going to derive transforms the whole set $\S(n)$ into the whole set $\S(N)$.

First of all, the equation
\begin{multline*}
u_{U_{k+1+P}}\cdots u_{U_{k+1+Q}-1}=\left(\varphi^{k+1+P}(0)\right)^{-1}\varphi^{k+1+Q}(0) \\
=\varphi\left(\left(\varphi^{k+P}(0)\right)^{-1}\varphi^{k+Q}(0)\right)=\varphi(u_{U_{k+P}}\cdots u_{U_{k+Q}-1})
\end{multline*}
implies that for any $J\in[U_{k+1+P},U_{k+1+Q}-1]$, there is a $j\in[U_{k+P},U_{k+Q}-1]$ such that 
$u_{J}\cdots u_{U_{k+1+Q}-1}$ is a suffix of $\varphi(u_{j}\cdots u_{U_{k+Q}-1})$.
Let $j$ be the greatest number with this property. The following triple,
$$
\left(\psi,a,b_{-L}\cdots b_0\cdots b_L\right):=
\left(
\Psirel(u_j u_{j+1}\cdots u_{j+n-1}),
u_j,
u_{j+n-L}\cdots u_{j+n}\cdots u_{j+n+L}
\right)
$$
is obviously an element of $\S(n)$.
Our aim is to express $\Psirel(u_J\cdots u_{J+N-1})$, $u_J$, $u_{J+N-L}\cdots u_{J+N+L}$ in terms of $(\psi,a,b_{-L}\cdots b_L)$.
Before we do so, it is useful to introduce symbols for the images of $a$ and $b_{-L}\cdots b_L$,
\begin{equation}\label{xy}
\begin{aligned}
\varphi(a)=x_1\cdots x_p\,,\qquad &\qquad\varphi(b_0)=y_1\cdots y_q\,, \\
\varphi(b_{-L}\cdots b_{-1})=y_{-r}\cdots y_0\,,&\qquad\varphi(b_{1}\cdots b_{L})=y_{q+1}\cdots y_{s}\,.
\end{aligned}
\end{equation}

Now we can proceed to expressing $\Psirel(u_J\cdots u_{J+N-1})$, $u_J$, $u_{J+N-L}\cdots u_{J+N+L}$.
We start with the term $u_J$.
Since $j$ is the \emph{greatest} number such that $u_{J}\cdots u_{U_{k+1+Q}-1}$ is a suffix of $\varphi(u_{j}\cdots u_{U_{k+Q}-1})$, necessarily
$u_J\cdots u_{J+N-1}$ is a prefix of
$$
\hat{x}^{-1}\varphi(u_j u_{j+1}u_{j+2}\cdots)\,,
$$
where $\hat{x}$ is a prefix of $\varphi(u_j)$ (i.e., of $\varphi(a)$) of length $t$ for a certain $t\in[0,|\varphi(a)|)$.
With regard to equation~\eqref{xy}, we have
\begin{equation}\label{a}
u_J=x_{t+1}\,.
\end{equation}
Both substitutions \eqref{simpleParry} and \eqref{nonsimpleParry} imply $\hat{x}=0^t$, hence $\hat{x}^{-1}=0^{-t}$.
The number $t$ is an ``offset'' parameter. There is an unambigous correspondence between $J$ and the pair $(j,t)$.

Now we proceed to the term $\Psirel(u_J u_{J+1}\cdots u_{J+N-1})$.
By definition,
$$
\Psirel(u_J u_{J+1}\cdots u_{J+N-1})=\Psi(u_J u_{J+1}\cdots u_{J+N-1})-\Psi(u_0 u_{1}\cdots u_{N-1})\,.
$$
It is easy to express the subtrahend $\Psi(u_0 u_{1}\cdots u_{N-1})$: due to equation~\eqref{prefix u} and the assumption $\langle N\rangle_U=\langle n\rangle_U d$, it holds
$$
u_0 u_{1}\cdots u_{N-1}=\varphi(u_0u_1\cdots u_{n-1})0^d\,.
$$
Therefore, with regard to equation~\eqref{matice},
\begin{equation}\label{subtrahend}
\Psi(u_0 u_{1}\cdots u_{N-1})=\Psi(u_0u_1\cdots u_{n-1})\M+\Psi(0^d)\,.
\end{equation}
Expressing the minuend $\Psi(u_J u_{J+1}\cdots u_{J+N-1})$ in terms of $(\psi,a,b_{-L}\cdots b_L)$ is a more complicated task.
Let us denote $\mathcal{L}:=|0^{-t}\varphi(u_j u_{j+1}\cdots u_{j+n-1})|$.
It holds:
\begin{itemize}
\item If $\mathcal{L}\leq N$, then $u_J\cdots u_{J+N-1}=0^{-t}\varphi(u_j\cdots u_{j+n-1})\tilde{y}$, where $\tilde{y}$ is a prefix of $\varphi(u_{j+n}u_{j+n+1}\cdots)=\varphi(b_{0}b_{1}\cdots)=y_1y_2\cdots$ of length $N-\mathcal{L}$.
\item If $\mathcal{L}\geq N$, then $u_J\cdots u_{J+N-1}=0^{-t}\varphi(u_j\cdots u_{j+n-1})\tilde{y}^{-1}$, where $\tilde{y}$ is a suffix of $\varphi(\cdots u_{j+n-2}u_{j+n-1})=\varphi(\cdots b_{-2}b_{-1})=\cdots y_{-1}y_{0}$ of length $\mathcal{L}-N$.
\end{itemize}
Hence
\begin{multline*}\label{minuend}
\Psi(u_J u_{J+1}\cdots u_{J+N-1})=\Psi\left(0^{-t}\varphi(u_j u_{j+1}\cdots u_{j+n-1})\right)+\sgn(N-\mathcal{L})\cdot\Psi(\tilde{y})\\
=-\Psi(0^t)+\Psi(u_j u_{j+1}\cdots u_{j+n-1})\M+\sgn(N-\mathcal{L})\cdot\Psi(\tilde{y})\,,
\end{multline*}
where we have again used equation~\eqref{matice}.
This result together with equation~\eqref{subtrahend} allows to express $\Psirel(u_J u_{J+1}\cdots u_{J+N-1})$,
\begin{equation}\label{relativni}
\begin{split}
&\Psirel(u_J u_{J+1}\cdots u_{J+N-1})=\Psi(u_J u_{J+1}\cdots u_{J+N-1})-\Psi(u_0 u_{1}\cdots u_{N-1}) \\
=&-\Psi(0^t)+\Psi(u_j u_{j+1}\cdots u_{j+n-1})\M+\sgn(N-\mathcal{L})\cdot\Psi(\tilde{y})-\Psi(u_0 u_{1}\cdots u_{n-1})\M-\Psi(0^d) \\
=&-\Psi(0^t)+\Psi^\mathrm{rel}(u_j u_{j+1}\cdots u_{j+n-1})\M-\Psi(0^d)+\sgn(N-\mathcal{L})\cdot\Psi(\tilde{y}) \\
=&-\Psi(0^{t+d})+\psi\M+\sgn(N-\mathcal{L})\cdot\Psi(\tilde{y})\,,
\end{split}
\end{equation}
where $\psi$ denotes $\Psi^\mathrm{rel}(u_j u_{j+1}\cdots u_{j+n-1})$.
Equation~\eqref{relativni} is not yet satisfactory because of the term $\sgn(N-\mathcal{L})\cdot\Psi(\tilde{y})$ that needs to be expressed in terms of $(\psi,a,b_{-L}\cdots b_L)$. The value $N-\mathcal{L}$ can be obtained by comparing the sums of components of vectors on the left- and right-hand side of~\eqref{relativni}. Let us denote the sum of components of $\psi\M$ by $h$. The quantity $N-\mathcal{L}$ is equal to the sum of components of $\sgn(N-\mathcal{L})\cdot\Psi(\tilde{y})$. Since the sum of components of any relative Parikh vector is $0$, equation~\eqref{relativni} implies
$$
0=-(t+d)+h+(N-\mathcal{L})\,,
$$
hence
\begin{equation}\label{N-L}
N-\mathcal{L}=t+d-h\,.
\end{equation}
With regard to above considerations, we have:
\begin{itemize}
\item If $t>h-d$, it holds $\mathcal{L}<N$, hence $\tilde{y}=y_1y_2\cdots y_{N-\mathcal{L}}=y_1y_2\cdots y_{t+d-h}$.
\item If $t<h-d$, it holds $\mathcal{L}>N$, hence $\tilde{y}=y_{N-\mathcal{L}+1}\cdots y_{-1}y_{0}=y_{t+d-h+1}\cdots y_{-1}y_{0}$.
\item If $t=h-d$, it holds $\mathcal{L}=N$, hence $\tilde{y}=\epsilon$.
\end{itemize}
This completes the search for the expression of $\Psirel(u_J u_{J+1}\cdots u_{J+N-1})$. It remains to express the third term of the triple, namely $u_{J+N-L}\cdots u_{J+N+L}$, in terms of $y_{-r}\cdots y_{s}$. It holds:
\begin{itemize}
\item If $t>h-d$, then $u_J\cdots u_{J+N-1}=0^{-t}\varphi(u_j\cdots u_{j+n-1})y_1y_2\cdots y_{t+d-h}$, thus $u_{J+N}=y_{t+d-h+1}$.
\item If $t<h-d$, then $u_J\cdots u_{J+N-1}=0^{-t}\varphi(u_j\cdots u_{j+n-1})(y_{t+d-h+1}\cdots y_{-1}y_{0})^{-1}$, thus $u_{J+N}=y_{t+d-h+1}$.
\item If $t=h-d$, then $u_J\cdots u_{J+N-1}=0^{-t}\varphi(u_j\cdots u_{j+n-1})$, thus $u_{J+N}=y_1=y_{t+d-h+1}$.
\end{itemize}
In all cases we have $u_{J+N}=y_{t+d-h+1}$, hence
$$
u_{J+N-L}\cdots u_{J+N}\cdots u_{J+N+L}=y_{t+d-h+1-L}\cdots y_{t+d-h+1} \cdots y_{t+d-h+1+L}\,.
$$
At this moment we have expressed all three elements of the triple
$$
(\Psirel(u_J\cdots u_{J+N-1}),u_J,u_{J+N-L}\cdots u_{J+N+L})
$$
in terms of $(\psi,a,b_{-L}\cdots b_L)$.
Note also that applying the formulas on any chosen $(\psi,a,b_{-L}\cdots b_L)\in\S(n)$ with any choice of the ``offset'' parameter $t\in[0,|\varphi(a)|)$ naturally gives a triple belonging to $\S(N)$.

In the rest of the proof we need to check that the subscripts of $y$ occurring in previous expressions do not run over the interval $[-r,s]$, i.e.,
$$
t+d-h\leq s\,,\quad t+d-h+1\geq-r\,,\quad t+d-h+1-L\geq-r\,,\quad t+d-h+1+L\leq s\,.
$$
This system of conditions is equivalent to
\begin{equation}\label{ineq s r}
s-L-1-(t+d-h)\geq0\quad\wedge\quad r-L+1+(t+d-h)\geq0\,.
\end{equation}
Let us estimate the left-hand sides of inequalities~\eqref{ineq s r}.
It holds $d\in\{0,1,\ldots,\alpha_0\}$, cf. Section~\ref{Preliminaries}. Since $t<|\varphi(a)|$ and $\max_{\ell\in\A}|\varphi(\ell)|=\alpha_0+1$, we have $t\leq\alpha_0$. It also holds $|h|\leq H$ due to equation~\eqref{H}. Consequently,
\begin{equation}\label{l-N}
-H\leq t+d-h\leq H+2\alpha_0\,.
\end{equation}
We also need to estimate $s,r$, for which we use the definition of $L$. Since $|\varphi(b_{1}\cdots b_{L})|=|y_{q+1}\cdots y_{s}|=s-q$, equation \eqref{L} implies $s-q-L\geq2\alpha_0+H$. Obviously $q\geq1$, hence $s-1-L\geq2\alpha_0+H$.
Similarly, equation~\eqref{L} together with $|\varphi(b_{-L}\cdots b_{-1})|=|y_{-r}\cdots y_0|=r+1$ implies $r+1-L\geq2\alpha_0+H$.
Combining these inequalities with inequalities~\eqref{l-N}, we obtain
\begin{align*}
s-L-1-(t+d-h)&\geq 2\alpha_0+H-(H+2\alpha_0)=0\,, \\
r-L+1+(t+d-h)&\geq 2\alpha_0+H-H=2\alpha_0\geq0\,,
\end{align*}
which proves inequalities~\eqref{ineq s r}.

\end{proof}

Let us summarize the algorithm for transforming $\S(n)$ into $\S(N)$, bringing together formulas derived in the proof of Proposition~\ref{inductive}.
Assume that the set $\S(n)$ for a certain $n\in\N$ is given, and $\langle N\rangle_U=\langle n\rangle_U d$ for a $d\in\{0,1,\ldots,\alpha_0\}$. According to the proof of Proposition~\ref{inductive}, the set $\S(N)$ can be constructed from $\S(n)$ by a procedure that consists in taking the elements $(\psi,a,b_{-L}\cdots b_L)\in\S(n)$ one by one, and for each of them performing the following steps:
\begin{enumerate}
\item Denote the sum of components of the vector $\psi\M$ by $h$, and define $x_i$, $y_i$ according to equations~\eqref{xy}.
\item For every $t=0,1,\ldots,|\varphi(a)|-1$, construct the triple
$$
\left(\hat{\Psi},
\hat{a},
\hat{b}_{-L}\cdots\hat{b}_{L}
\right)\,,
$$
where $\hat{a}=x_{t+1}$, $\hat{b}_{-L}\cdots\hat{b}_{L}=y_{t+d-h+1-L}\cdots y_{t+d-h+1}\cdots y_{t+d-h+1+L}$, and
$$
\hat{\Psi}=\left\{\begin{array}{ll}
\psi\M-\Psi(0^{t+d}) & \text{if } t=h-d; \\
\psi\M-\Psi(0^{t+d})+\Psi(y_1\cdots y_{t+d-h}) & \text{if } t>h-d; \\
\psi\M-\Psi(0^{t+d})-\Psi(y_{1+t+d-h}\cdots y_{0}) & \text{if } t<h-d.
\end{array}\right.
$$
\end{enumerate}
The collection of all triples $\left(\hat{\Psi},\hat{a},\hat{b}_{-L}\cdots\hat{b}_{L}\right)$ constructed in step 2 constitutes the set $\S(N)$.

\subsection{$U$-automaticity}

According to Proposition~\ref{inductive}, if the normal $U$-representation of an $N\in\N$ is obtained as the normal $U$-representation of an $n\in\N$ with an attached digit, then the set $\S(N)$ can be constructed from the set $\S(n)$. Recall also that we have proven in Corollary~\ref{Finite} that there exist finitely many sets $\S_1,\ldots,\S_M$ such that for any $n\in\N$, $\S(n)$ coincides with $\S_j$ for a certain $j$. These two properties together have a straightforward and important corollary:

\begin{corollary}\label{Subsets}
There exists a function $\delta(j,d)$ for $j\in\{1,\ldots,M\}$ and $d\in\{0,\ldots,\alpha_0\}$ such that for any pair $n,N\in\N$ satisfying
$$
\langle n\rangle_U=d_k d_{k-1}\cdots d_1d_0 \qquad\text{and}\qquad \langle N\rangle_U=d_k d_{k-1}\cdots d_1d_0d
$$
it holds
\begin{equation}\label{SNn}
\S(n)=\S_j \qquad\Rightarrow\qquad \S(N)=\S_{\delta(j,d)}\,.
\end{equation}
\end{corollary}
We may assume without loss of generality that the sets $\S_j$ are enumerated so that
\begin{equation}\label{S_d}
\S_d=\S(d) \qquad\text{for all } d=1,\ldots,\alpha_0\,.
\end{equation}
For such enumeration, we put formally
\begin{equation}\label{S_0}
\S_0=\S(0):=\emptyset
\end{equation}
and extend the definition of $\delta$ to the value $j=0$ as follows,
\begin{equation}\label{delta(0,d)}
\delta(0,d):=d \qquad\text{for all } d=1,\ldots,\alpha_0\,.
\end{equation}
The assumptions~\eqref{S_d}, \eqref{S_0} and \eqref{delta(0,d)} make the implication~\eqref{SNn} valid also for pairs $n,N$ such that $n=0$ and $N\in\{1,\ldots,\alpha_0\}$.

The function $\delta$ allows to determine $\S(n)$ for any $n\in\N$, as it is demonstrated in Proposition~\ref{S_j}. Let us recall that the symbol $\delta(0,d_k d_{k-1}\cdots d_1d_0)$ has the meaning
$$
\delta(0,d_k d_{k-1}\cdots d_1d_0)\equiv\delta(\delta(\cdots\delta(\delta(0,d_k),d_{k-1})\cdots,d_1),d_0)\,,
$$
cf. Section~\ref{Preliminaries}.
\begin{proposition}\label{S_j}
Let $n\in\N$. It holds
\begin{equation}\label{S(n) S_j}
\S(n)=\S_j \quad\text{for}\quad j=\delta(0,\langle n\rangle_U)\,.
\end{equation}
\end{proposition}

\begin{proof}
We prove the statement by induction on $k$.
\begin{itemize}
\item[I.] Let $k=0$. Then $\langle n\rangle_U=d_0\in\{0,1,\ldots,\alpha_0\}$ and $j=\delta(0,{d_0})$. On one hand, $\langle n\rangle_U=d_0$ implies $n=d_0$, hence $\S(n)=\S(d_0)$, thus $\S(n)=\S_{d_0}$ by assumptions~\eqref{S_d}. On the other hand, it holds $\delta(0,{d_0})=d_0$ due to definition~\eqref{delta(0,d)}, hence $\S_j=\S_{d_0}$. To sum up, the statement $\S(n)=\S_j$ for $j=\delta(0,d_0)$ holds true.
\item[II.] Let $\langle n \rangle_U =d_k d_{k-1}\cdots d_1d_0$ have $k$ digits for a $k\geq1$. We assume that equation~\eqref{S(n) S_j} holds true for any integer with a normal $U$-representation having $k-1$ digits, in particular for $\langle n'\rangle_U=d_k d_{k-1}\cdots d_1$. Equation~\eqref{S(n) S_j} gives $\S(n')=\S_{j'}$ for $j'=\delta(0,d_kd_{k-1}\cdots d_2d_1)$. Due to equation~\eqref{SNn} it holds $\S(n)=\S_{\delta(j',d_0)}$. With regard to the expression for $j'$, we have $\delta(j',d_0)=\delta(\delta(0,d_kd_{k-1}\cdots d_2d_1),d_0)=\delta(0,d_kd_{k-1}\cdots d_2d_1d_0)$. If we denote this value by $j$, we see that equation~\eqref{S(n) S_j} holds true.
\end{itemize}
\end{proof}

Let us define sets $\Prel_1,\ldots,\Prel_M$ as follows,
$$
\Prel_j=\left\{\psi\;;\;(\psi,a,b_{-L}\cdots b_L)\in\S_j\right\} \qquad\text{for all $j=1,\ldots,M$}\,.
$$
Taking advantage from Proposition~\ref{S_j}, we can express sets $\Prel(n)$ for $n\in\N$ in terms of $\Prel_1,\ldots,\Prel_M$ and $\langle n\rangle_U$,
\begin{equation}\label{Prel(n) Prel_j}
\Prel(n)=\Prel_{\delta(0,\langle n\rangle_U)}\,.
\end{equation}
Consequently, there exists a finite number of sets of relative Parikh vectors, $\Prel_1,\ldots,\Prel_M$, such that for any $n\in\N$, $\Prel(n)$ is equal to $\Prel_j$ for a certain $j\in\{1,\ldots,M\}$. The value $j$ can be found using the normal $U$-representation of $n$, therefore, in $\mathcal{O}(\log n)$ steps.

Recall that the abelian complexity $\AC(n)$ is equal to the cardinality of the set $\Prel(n)$, cf. equation~\eqref{ACrel}. With regard to that, we introduce a function $\tau:\{1,\ldots,M\}\to\N$ by the relation
\begin{equation}\label{tau}
\tau(j)=\#\Prel_j\,.
\end{equation}
Now we take advantage of equation~\eqref{Prel(n) Prel_j} to obtain the final formula for $\AC(n)$:
\begin{theorem}\label{Final}
The abelian complexity of $\u$ is given by the formula
\begin{equation}\label{AC final}
\AC(n)=\tau\left(\delta(0,\langle n\rangle_U)\right)\,.
\end{equation}
\end{theorem}
\begin{proof}
The statement is a straightforward consequence of equations~\eqref{ACrel}, \eqref{Prel(n) Prel_j} and the definition~\eqref{tau}.
\end{proof}
Equation~\eqref{AC final} implies that the sequence $\left(\AC(n)\right)_{n=1}^{\infty}$ is $U$-automatic. The sequence is generated by a DFAO $\left(Q,\Sigma,\delta,q_0,\Delta,\tau\right)$, where
\begin{itemize}
\item $S=\{0,1,\ldots,M\}$ is the set of states;
\item $\Sigma=\{0,1,\ldots,\alpha_0\}$ is the input alphabet;
\item $\delta$ is the transition function, determined in Corollary~\ref{Subsets};
\item $q_0=0$ is the initial state;
\item $\Delta\subset\{1,\ldots,(2c+1)^m\}$ is the output alphabet, corresponding to the image of $\AC$;
\item $\tau$ is the output function, defined above by~\eqref{tau}.
\end{itemize}

%%%%%%%%%%%%%%%%%%%%%%%%%%%%%%%%%%%%%%%%%%%%%%%%%%%%%%%%%%%%%%%%%%%%%%%%%%%%%%%%%%%%%%%%%%%%%%

\section{How to find $\S_1,\ldots,\S_M$}

The proof of existence of a deterministic finite automaton generating the sequence $\left(\AC(n)\right)_{n=1}^{\infty}$, given in Section~\ref{Main}, relies on the existence of sets $\S_1,\ldots,\S_M$ with properties referred to as (P1), (P2) and (P3). In this section we will present an algorithm allowing to find the sets $\S_1,\ldots,\S_M$ explicitly.

Recall that $\S(n)$ for any $n\in\N$ can be found using equation~\eqref{S(n) p q}, thus one might attempt to construct sets $\S(n)$ for $n=1,2,3,\ldots$ and enumerate those that are mutually different by $\S_j$ for $j=1,2,\ldots$. %Sets $\S_j$ obtained in this way naturally have the properties (P1) and (P3), but on the other hand, 
However, this approach does not work, because it provides no criterion to recognize when the collection of sets $\S_j$ is already complete in the sense of property (P2). Below we propose a procedure with a stop criterion, which yields $\S_1,\ldots,\S_M$ in a finite number of steps.

The stop criterion is based on exploiting property (P3). We know that if two numbers $n,N$ satisfy $\langle N\rangle_U=\langle n\rangle_U d$ for a certain $d\in\{0,1,\ldots,\alpha_0\}$, then $\S(N)$ depends only on $\S(n)$ and $d$. Note that the condition $\langle N\rangle_U=\langle n\rangle_U d$ means that $\langle N\rangle_U$ has one digit more than $\langle n\rangle_U$; in other words, $n\in[U_{k},U_{k+1})$ and $N\in[U_{k+1},U_{k+2})$ for a certain $k\in\N_0$. Therefore, for any $k\in\N_0$, the collection of sets
$$
\S(N) \qquad\text{for}\; N=U_{k+1},U_{k+1}+1,\ldots,U_{k+2}-1
$$
is obtained from the collection of pairs
$$
(\S(n),d) \qquad\text{for}\; n=U_k,U_k+1,\ldots,U_{k+1}-1\;\wedge\;\langle n\rangle_U d \text{ is a valid normal $U$-representation}\,.
$$
Roughly speaking, our algorithm consists in constructing \emph{pairs} $(\S(n),d)$ of certain specified properties for $n\in[U_{k},U_{k+1})$ with $k=0,1,2,\ldots$, and the stop criterion is defined as attaining a $k$ for which no new pairs $(\S(n),d)$ are found. The idea is that if all pairs $(\S(n),d)$ constructed for $n\in[U_{k},U_{k+1})$ coincide with pairs already constructed for $n'<U_{k}$, then increasing $k$ by one and repeating the procedure cannot give any novel output, simply because it has no novel input. This means that all possible pairs $(\S(n),d)$ have been already found, thus all possible sets $\S(n)$ have been found as well. Since conditions imposed on $\S(n)$ and $d$ (see below) imply that there exists a finite number of pairs $(\S(n),d)$, the procedure necessarily terminates after a finite number of steps.

A more detailed description of the algorithm follows. Note that the algorithm in the present form is schematic and by far not optimal; its aim is primarily to be simple.
\begin{itemize}
\item[1.] Construct $\S(n)$ for $\langle n\rangle_U=d_0$ with $d_0=1,\ldots,\alpha_0$. Define $\mathrm{Mem}_1=\{\S(1),\ldots,\S(\alpha_0)\}$ and 
$$
\mathrm{Mem}_2=\{(\S(n),d)\,;\,n=1,\ldots,\alpha_0,d=0,1,\ldots,\alpha_0\}\,.
$$
\item[2.] For every $\S(n)$ added to $\mathrm{Mem}_1$ in the previous step and for every $d=0,1,\ldots,\alpha_0$,
\begin{itemize}
\item[i.] apply the procedure from the end of Section~\ref{Sect. (P3)} formally without checking whether $\langle n\rangle_U d$ is a valid normal $U$-representation, and denote the resulting set by $\S(N)$;
\item[ii.] add the pair $(\S(N),d)$ into $\mathrm{Mem}_2$;
\item[iii.] if all entries of all vectors $\psi$ in $(\psi,a,b_{-L}\cdots b_L)\in\S(N)$ have absolute values bounded by $c$, add $\S(N)$ into $\mathrm{Mem}_1$.
\end{itemize}
\item[3.] Repeat step 2 as long as the cardinality of $\mathrm{Mem}_2$ grows. Once it stops growing, put $\{\S_1,\ldots,\S_M\}=\mathrm{Mem}_1$.
\end{itemize}
We conclude the section by several explanatory remarks.
\begin{itemize}
\item The set $\mathrm{Mem}_2$ contains pairs ``(set of type $\S(n)$, $d$)'' for all $d\in\{0,1,\ldots,\alpha_0\}$. The set $\mathrm{Mem}_1$ contains sets of type $S(n)$ found by the procedure that satisfy an additional condition, formulated in step 2iii.
\item The condition ``all entries of all vectors $\psi$ in $(\psi,a,b_{-L}\cdots b_L)\in\S(N)$ have absolute values bounded by $c$'' is used in the algorithm in place of the condition ``$\langle n\rangle_U d$ is a valid normal $U$-representation'' for the reason that its fulfilment does not depend directly on $n$. Note that the conditions are not equivalent: the former one is obviously weaker than the latter one. Nevertheless, the former condition, used in step 2iii, still guarantees that $\mathrm{Mem}_1$ and $\mathrm{Mem}_2$ are finite. Indeed, the finiteness of $\mathrm{Mem}_1$ can be proven by the same method that has been used in Section~\ref{Sect. (P2)}, and the bounded cardinality of $\mathrm{Mem}_2$ is a straightforward consequence, because $\mathrm{Mem}_2$ is constructed from elements of $\mathrm{Mem}_1$.
\item The use of the weakened condition in step 2iii requires to increase the constant $L$ in order to avoid an overflow of subscripts of $y_i$, cf. the end of the proof of Proposition~\ref{inductive}. It is sufficient to choose $H$ and $L$ according to Remark~\ref{HL}.
\item Since the algorithm does not check validities of normal $U$-representations, the set $\mathrm{Mem}_1$ may contain elements that do not correspond to $\S(n)$ for any $n\in\N$. However, the presence of extra elements do not pose a problem, as they do not break the finiteness of $\mathrm{Mem}_1$.
\end{itemize}

%%%%%%%%%%%%%%%%%%%%%%%%%%%%%%%%%%%%%%%%%%%%%%%%%%%%%%%%%%%%%%%%%%%%%%%%%%%%%%%%%%%%%%%%%%%%%%

\section{Generalizations}\label{Generalizations}

In Section~\ref{Main}, we have found a finite number of sets $\Prel_1,\ldots,\Prel_M$ such that for any $n\in\N$, the set of relative Parikh vectors $\Prel(n)$ is equal to a certain $\Prel_j$. According to equation~\eqref{Prel(n) Prel_j}, the assignment of $j$ to a given $n\in\N$ can be done by a finite automaton using the transition function $\delta$.
Consequently, any function $F:\N\to\N$ that is defined in terms of the set of relative Parikh vectors $\Prel(n)$ can be evaluated by a finite automaton using the transition function $\delta$ and an appropriate output function $\tau$. The output function $\tau$ reflects the function $F$. For instance, in previous sections we focused on the abelian complexity; since the function $\AC$ is defined in the way $\AC(n)=\#\Prel(n)$, the corresponding output function has been taken in the form $\tau(j)=\#\Prel_j$ cf. equation~\eqref{tau}. Let us bring in another example. The \emph{balance function}~\cite{Ad02,BT} of a word $\u$ is defined as
$$
B_\u(n)=\max\{\left|\,|w|_a-|w'|_a\right|\,;\,a\in\A, \text{$w,w'$ are factors of $\u$, $|w|=|w'|=n$}\}\,.
$$
The right hand side can be rewritten in terms of maximum norms of Parikh vectors,
$$
B_\u(n)=\max\{\|\Psi(w)-\Psi(w')\|_\infty\,;\,\text{$w,w'$ are factors of $\u$, $|w|=|w'|=n$}\}.
$$
A simple manipulation leads to
$$
\Psi(w)-\Psi(w')=(\Psi(w)-\Psi(u_0u_1\cdots u_{n-1}))-(\Psi(w')-\Psi(u_0u_1\cdots u_{n-1}))=\Psirel(w)-\Psirel(w')\,.
$$
It allows us to define the balance function in terms of the set $\Prel(n)$,
$$
B_\u(n)=\max\{\|\psi-\psi'\|_\infty\,;\,\psi,\psi'\in\Prel(n)\}\,.
$$
Let us put
\begin{equation}\label{tau_B}
\tau_B(j):=\max\left\{\|\psi-\psi'\|_\infty\,;\,\psi,\psi'\in\Prel_j\right\}
\end{equation}
for all $j=1,\ldots,M$.
Then it holds
$$
B_\u(n)=\tau_B\left(\delta(0,\langle n\rangle_U)\right)\,.
$$
The result is analogical to equation~\eqref{AC final} of Theorem~\ref{Final}. Consequently, the balance function of a balanced Parry word can be evaluated by a DFAO $\left(Q,\Sigma,\delta,q_0,\Delta,\tau_B\right)$, where $Q,\Sigma,\delta,q_0,\Delta$ have exactly the same meanings as in the DFAO described at the end of Section~\ref{Main}, and $\tau_B$ is given by equation~\eqref{tau_B}.

A similar result can be obtained for any other function given in terms of the set of relative Parikh vectors $\Prel(n)$.

\section*{Acknowledgements}

The author is thankful to J.-P. Allouche for useful comments and suggestions.

%%%%%%%%%%%%%%%%%%%%%%%%%%%%%%%%%%%%%%%%%%%%%%%%%%%%%%%%%%%%%%%%%%%%%%%%%%%%%%%%%%%%%%%%%%%%%%

\end{document}